\documentclass{amsart}
\newtheorem{theorem}{Theorem}[section]
\newtheorem{corollary}[theorem]{Corollary}
\newtheorem{lemma}[theorem]{Lemma}
\newtheorem{proposition}[theorem]{Proposition}
\newtheorem{example}[theorem]{Example}

\newtheorem{remark}[theorem]{Remark}
\newtheorem{definition}[theorem]{Definition}

\begin{document}

\title[Upper and lower Lebesgue classification]{Upper and lower Lebesgue classification of multivalued mappings of two variables}

\author{Olena Karlova}

\address{Chernivtsi National University, Department of Mathematical Analysis,
Kotsjubyns'koho 2, Chernivtsi 58012, Ukraine}\email{maslenizza.ua@gmail.com}

\author{Volodymyr Mykhaylyuk}
\address{Chernivtsi National University, Department of Mathematical Analysis,
Kotsjubyns'koho 2, Chernivtsi 58012, Ukraine}
\email{vmykhaylyuk@ukr.net}



\begin{abstract} We introduce a functional Lebesgue classification of multivalued mappings and obtain results on upper and lower Lebesgue classifications of multivalued mappings $F:X\times Y\to Z$ for wide classes of spaces $X$, $Y$ and $Z$.

\end{abstract}

\maketitle
\begin{center}
{\it Dedicated to the memory of Yu.B.~Zelinskii}
\end{center}

\bigskip

\section{Introduction}

Investigations of Lebesgue classification of separately continuous single-valued functions (i.e., functions of several variables which are continuous with respect to each variable) and theirs analogs were started by H.~Lebesgue in \cite{L} and by K.~Kuratowski in \cite{Kur}. These investigations were continued in papers of many mathematicians (see, for example, \cite{Mon}, \cite{Kur1}, \cite{Ban}, \cite{Bur}, \cite{MMM}, \cite{K}   and the literature given there).

Some analogs of Lebesgue classification are also known for multivalued maps and are connected with the upper and lower semi-continuity.
Namely, a multivalued map $F:X\to [0,1]$ defined on a topological space $X$ is called {\it upper (lower) semi-continuous at $x_0\in X$} if for every open  set $U$  in $[0,1]$ with $F(x_0)\subseteq U$ ($F(x_0)\cap U\ne\emptyset$) the set
$$
F^+(U)=\{x\in X:F(x)\subseteq U\}
$$
$$
(F^-(U)=\{x\in X:F(x)\cap U\ne\emptyset\})
$$
is a neighborhood of $x_0$ in $X$. A multivalued map $F:X\to [0,1]$ is {\it continuous at $x_0\in X$} if it is  upper and lower semi-continuous at $x_0$ simultaneously. It is well known that a multivalued map $F:X\to [0,1]$ is continuous at $x_0\in X$ if and only if it is continuous at $x_0$ as a single-valued map with values in the space of all non-empty subsets of $[0,1]$ with the Vietoris topology.

For topological spaces $X$ and $Y$ let ${\rm U}(X,Y)$ (${\rm L}(X,Y)$) stands for the collection of all upper (lower) semi-continuous multivalued maps $F:X\to Y$.

Let $X$ and $Y$ be topological spaces and $\alpha$ be an at most countable ordinal. A multivalued map $F:X\to Y$ belongs to the
\begin{enumerate}
\item[-] {\it $\alpha$'th upper Lebesgue  class} if for every open set $A\subseteq Y$ the set $F^+(A)$ belongs to the $\alpha$'th  additive class  in $X$;
\item[-] {\it $\alpha$'th lower Lebesgue  class} if for every open set $A\subseteq Y$ the set $F^-(A)$  belongs to the $\alpha$'th  additive class  in $X$.
\end{enumerate}
Let us observe that Lebesgue classes are often called Borel classes.

For topological spaces $X$ and $Y$ the collection of all multivalued maps $F:X\to Y$ of the $\alpha$'th upper (lower) Lebesgue class   we denote by ${\rm U}_\alpha(X,Y)$ (${\rm L}_\alpha(X,Y)$).

For a multivalued map $F:X\times Y\to Z$, $x\in X$ and $y\in Y$ we write  $F^x(y)=F_y(x)=F(x,y)$.

Recall that a topological space is said to be {\it perfect} if every its closed subset is $G_\delta$.

G.~Kwieci\'{n}ska  proved in \cite{Kwi} the following result on Lebesgue classification of multivalued maps of two variables.

\begin{theorem}[\cite{Kwi}]\label{th:0.1}
Let  $(X,d)$ be a metric space, $\mathcal T$ be a topology on $X$, $D\subseteq X$ be an at most countable set, $(U(x):x\in X)$ be a family of ${\mathcal T}$-open sets $U(x)\subseteq X$, $Y$ and $Z$  be perfectly normal spaces, $\alpha$ be an at most countable ordinal and $F:X\times Y\to Z$ be a compact-valued (multivalued) map satisfying the conditions
\begin{enumerate}
\item[(a)] $D$ is dense in $(X,{\mathcal T})$;

\item[(b)] for every $x\in D$ the set $A(x)=\{u\in X:x\in U(u)\}$ belongs to the $\alpha$'th additive class in $(X,d)$;

\item[(c)] for every $x\in X$ the sequence $(B_n(x))_{n\in\omega}$ of sets $$B_n(x)=U(x)\cap \{u\in X:d(x,u)<\tfrac1n\}$$ forms a base of $(X,{\mathcal T})$ at $x$;

\item[(d)] for every $y\in Y$ the multivalued map $F_y:(X,{\mathcal T})\to Z$ is continuous;

\item[(e)] for every $x\in D$, the multivalued map $F^x:Y\to Z$ belongs to the $\alpha$'th lower (upper) Lebesgue class.
\end{enumerate}
Then $F$ is of the upper (lower) Lebesgue class $\alpha+1$ on the product $(X,d)\times Y$.
\end{theorem}

Another variants of Lebesgue classification of multivalued maps of two variables were obtained in \cite{KS}.

\begin{theorem}[\cite{KS}]\label{th:0.2}
Let  $X$ be a metrizable space, $D$ be a dense subset of $X$, $Y$ be a perfect space, $Z$  be a perfectly normal space, $\alpha$ be an at most countable ordinal and $F:X\times Y\to Z$ be a compact-valued (multivalued) map satisfying the conditions
\begin{enumerate}
\item[(a)] for every $y\in Y$ the multivalued map $F_y:X\to Z$ is continuous;

\item[(b)] for every $x\in D$, the multivalued map $F^x:Y\to Z$ is of the $\alpha$'th lower (upper) Lebesgue class.
\end{enumerate}
Then $F$ is of the  upper (lower) Lebesgue class $\alpha+1$ on the product $X\times Y$.
\end{theorem}

In this paper we introduce a functional Lebesgue classification of multivalued maps and generalize Theorems \ref{th:0.1} and \ref{th:0.2} for a wide class of topological spaces $X$.

\section{multivalued maps of the $\alpha$'th upper and lower Lebesgue functional class}

\begin{definition}\label{def:1} {\rm
  Let $X$ and $Y$ be topological spaces. A multivalued map $F:X\to Y$ is called
  \begin{enumerate}
  \item[-] {\it upper (lower) functionally semi-continuous multivalued map} if $F^+(A)$ ($F^-(A)$) is a functionally open set for every functionally open set $A\subseteq Y$;
  \item[-] {\it strongly upper (lower) functionally semi-continuous multivalued map} if $F^+(A)$ ($F^-(A)$) is a functionally open set for every open set $A\subseteq Y$;
  \item[-] {\it weakly upper (lower) functionally semi-continuous multivalued map} if $F^+(A)$ ($F^-(A)$) is an open set for every functionally open set $A\subseteq Y$.
  \end{enumerate}
 }
\end{definition}

The collection of all upper (lower, strongly upper, strongly lower, weakly upper, weakly lower) functional semi-continuous multivalued maps $F:X\to Y$ we denote by ${\rm U}^f(X)$ (${\rm L}^f(X)$, ${\rm U}_s^f(X)$, ${\rm L}_s^f(X)$, ${\rm U}_w^f(X)$, ${\rm L}_w^f(X)$).

Let $X$ be a topological space. We denote by ${\mathcal A}_0(X)$ and ${\mathcal M}_0(X)$ the systems of all functionally open and functionally closed subsets of $X$, respectively. For every at most countable ordinal $\alpha\geq 1$ the system of all unions $\bigcup_{n\in\omega} A_n$ of sets $A_n$ from $\bigcup_{\xi<\alpha}{\mathcal M}_\alpha(X)$ we denote by ${\mathcal A}_\alpha (X)$, and the system of all intersections $\bigcap_{n\in\omega} M_n$ of sets $M_n$ from $\bigcup _{\xi<\alpha}{\mathcal A}_\alpha(X)$ we denote by ${\mathcal M}_\alpha (X)$. Clearly,
$$
{\mathcal A}_\alpha (X)=\{X\setminus M:M\in {\mathcal M}_\alpha (X)\}.
$$

\begin{definition}\label{def:2} {\rm
  Let $X$ and $Y$ be topological spaces and $\alpha$ be an at most countable ordinal. A multivalued map $F:X\to Y$ is called
  \begin{enumerate}
  \item[-] {\it upper Lebesgue functional class $\alpha$ multivalued map} if $F^+(A)\in {\mathcal A}_\alpha(X)$ for every functionally open set $A\subseteq Y$;
  \item[-] {\it lower Lebesgue functional class $\alpha$ multivalued map} if $F^-(A)\in {\mathcal A}_\alpha(X)$ for every functionally open set $A\subseteq Y$.
  \end{enumerate}    }
\end{definition}

It is easy to see that $F$ is an upper (lower) Lebesgue functional class $\alpha$ multivalued map if and only if $F^-(B)\in {\mathcal A}_\alpha(X)$ ($F^+(B)\in {\mathcal A}_\alpha(X)$) for every functionally closed set $B\subseteq Y$.

For topological spaces $X$ and $Y$ the collection of all multivalued maps $F:X\to Y$ of the $\alpha$'th upper (lower) Lebesgue functional class    we denote by ${\rm U}^f_\alpha(X)$ (${\rm L}^f_\alpha(X)$). Notice that ${\rm U}^f_0(X)={\rm U}^f(X)$ and ${\rm L}^f_0(X)={\rm L}^f(X)$.

The facts below follow  easily from the definitions and we omit their proof.

\begin{proposition}\label{properties}
Let $X$ and $Y$ be a topological spaces, $F:X\to Y$ be a multivalued map and $\alpha$ be at most countable ordinal.
\begin{enumerate}
\item ${\rm U}(X,Y)\cup {\rm U}^f(X,Y)\subseteq {\rm U}_w^f(X,Y)$ and ${\rm L}(X,Y)\cup {\rm L}^f(X,Y)\subseteq {\rm L}_w^f(X,Y)$.

\item ${\rm U}_s^f(X,Y)\subseteq {\rm U}(X,Y)\cap {\rm U}^f(X,Y)$ and ${\rm L}_s^f(X,Y)\subseteq {\rm L}(X,Y)\cap {\rm L}^f(X,Y)$.

\item If\, $X$ is perfectly normal, then  ${\rm U}_s^f(X,Y)={\rm U}(X,Y)\subseteq {\rm U}^f(X,Y)={\rm U}_w^f(X,Y)$ and ${\rm L}_s^f(X,Y)={\rm L}(X,Y)\subseteq {\rm L}^f(X,Y)={\rm L}_w^f(X,Y)$.

\item If\,\, $Y$ is completely regular, then ${\rm U}^f(X,Y)\subseteq {\rm U}(X,Y)$ and ${\rm L}^f(X,Y)\subseteq {\rm L}(X,Y)$.

\item If\, $Y$ is perfectly normal, then  ${\rm U}_s^f(X,Y)={\rm U}^f(X,Y)\subseteq {\rm U}(X,Y)={\rm U}_w^f(X,Y)$ and ${\rm L}_s^f(X,Y)={\rm L}^f(X,Y)\subseteq {\rm L}(X,Y)={\rm L}_w^f(X,Y)$.

\item  ${\rm U}^f_\alpha(X,Y)\subseteq {\rm L}^f_{\alpha+1}(X,Y)$.

\item If $F$ is a compact-valued map then ${\rm L}^f_\alpha(X,Y)\subseteq {\rm U}^f_{\alpha+1}(X,Y)$.
\end{enumerate}
\end{proposition}

\begin{proposition}\label{pr:1}
Let $Y$ be a topological space such that $\{\emptyset, Y\}$ is the system of all functionally open sets in $Y$ (see, for example, \cite[2.7.18]{Eng}). Then
\begin{enumerate}
\item for every topological space $X$ any multivalued map $F:X\to Y$ is upper and lower functionally semi-continuous;

\item for every $T_1$-space $Z$ any strongly upper functionally semi-continuous map $F:Y\to Z$ is constant;

\item for every (completely) regular space $Z$  any strongly lower (lower) functionally semi-continuous closed-valued map $F:Y\to Z$ is constant.
\end{enumerate}
\end{proposition}

\begin{proof}
$(1)$. Since $F^+(\emptyset)=F^-(\emptyset)=\emptyset$ and $F^+(Y)=F^-(Y)=X$, $F:X\to Y$ is upper and lower functionally semi-continuous.

$(2)$. Let $Z$ be a $T_1$-space, $F:Y\to Z$ be a map which is non-constant. We chose $y_1,y_2\in Y$ such that $F(y_1)\not\subseteq F(y_2)$. Since $Y$ is a $T_1$-space, there exists an open set $G\subseteq Z$ such that $F(y_1)\not\subseteq G\supseteq F(y_2)$. Then $y_1\not\in F^+(G)\ni y_2$. Therefore, $F^+(G)\not\in\{\emptyset, Y\}$ and $F^+(G)$ is not functionally open. Hence, $F$ is not strongly upper functionally semi-continuous.

$(3)$. Let $Z$ be a regular space, $F:Y\to Z$ be a non-constant map. We choose $y_1,y_2\in Y$ such that $F(y_1)\not\subseteq F(y_2)$. Since $Y$ is a regular space and $F(y_2)$ is closed, there exists an open set $G\subseteq Z$ such that $G\cap F(y_1)\ne\emptyset$ and $G\cap F(y_2)=\emptyset$. Then $y_1\in F^-(G)\not\ni y_2$. Therefore, $F^-(G)\not\in\{\emptyset, Y\}$ and $F^-(G)$ is not functionally open. Hence, $F$ is not strongly lower functionally semi-continuous. In the case of completely regular space $Z$ we can choose a functionally open set $G$ and the map $F$ is not lower functionally semi-continuous.
\end{proof}

\begin{example}\label{ex:1}
Let $A\subseteq[0,1]$ be a non Borel-measurable set. Then the multivalued map $F:[0,1]\to[0,1]$,
 \begin{equation*}
 F(x)=\left\{\begin{array}{ll}
                        [0,1], & x\in A\\
                        {[}0,1), & x\in [0,1]\setminus A,
                       \end{array}
 \right.
 \end{equation*}
is a lower (functionally) semi-continuous and $F$ is upper (functionally) nonmeasurable, that is, $F\not\in \bigcup_{\alpha<\omega_1}{\rm U}^f_\alpha([0,1],[0,1])$.

\end{example}

\section{Functional Lebesgue classification of multivalued maps of two variables}

The following auxiliary statement can be found in \cite[Proposition 1.4]{K}.

\begin{lemma}\label{lem:1.1} Let $X$ be a topological space, $\alpha$ be an at most countable ordinal, $(U_i:i\in I)$ be a locally finite family of functionally open in $X$ sets $U_i$ and $(A_i:i\in I)$ be a family of sets $A_i\in {\mathcal A}_\alpha(X)$ ($A_i\in {\mathcal M}_\alpha(X)$) with $A_i\subseteq U_i$ for every $i\in I$. Then $\bigcup_{i\in I}A_i\in {\mathcal A}_\alpha(X)$ ($\bigcup_{i\in I}A_i\in {\mathcal M}_\alpha(X)$).
\end{lemma}

Let us observe that a union of a locally finite family of sets of the $\alpha$'th functionally  multiplicative class does not belong to the same class even for $\alpha=0$.

Indeed, let $X$ be the Niemytski plane, i.e. $X=\mathbb R\times [0,+\infty)$, where a base of neighborhoods of $(x,y)\in X$ with $y>0$ form open balls with the center in $(x,y)$, and a base of neighborhoods of $(x,0)$ form the sets $U\cup\{(x,0)\}$, where $U$ is an open ball which tangent to $\mathbb R\times \{0\}$ in the point $(x,0)$.

Notice  that for every  $p\in X$ the set $\{p\}$ is functionally closed in $X$, since each continuous function on $\mathbb R\times[0,+\infty)$ is continuous on $X$. Then the family $\mathcal F=(\{(x,0)\}: x\in\mathbb Q)$ consists of functionally closed subsets of $X$. To obtain a contradiction, we assume that the union $F=\bigcup \mathcal F$ is functionally closed in $X$ and choose a continuous function $f:X\to [0,1]$ such that $F=f^{-1}(0)$.
For all $(x,y)\in X$ and $n\in\mathbb N$ we put
        $$
        f_{n}(x,y)=\left\{\begin{array}{ll}
                              f(x,y), & y\ge\frac 1n,\\
                              f(x,\frac 1n),  & 0\le y<\frac 1n.
                            \end{array}
        \right.
        $$
        Then $f_{n}:\mathbb R\times [0,+\infty)\to [0,1]$ is a continuous function and $\lim_{n\to\infty}f_{n}(x,y)=f(x,y)$ for every $(x,y)\in X$.
        Since $F=\bigcap_{k=1}^\infty\bigcup_{n=k}^\infty f_n^{-1}([0,\frac 1k))$, $F$ is a $G_\delta$-subset of $\mathbb R\times[0,+\infty)$, which implies a contradiction.

\begin{definition}\label{def:3} {\rm A family $(A_i:i\in I)$ of subsets $A_i$ of topological space $X$ is called {\it functionally locally finite in $X$}, if there exists a locally finite in $X$ family $(U_i:i\in I)$ of functionally open in $X$ sets $U_i\supseteq A_i$. A family $(A_i:i\in I)$ of subsets $A_i$ of topological space $X$ is called {\it $\sigma$-functionally locally finite} if there exists a partition $I=\bigsqcup_{n\in\omega}I_n$ such that every family $(A_i:i\in I_n)$ is functionally locally finite in $X$.}
\end{definition}

\begin{theorem}\label{th:1.2}
Let  $X$, $Y$ and $Z$  be topological spaces, $\alpha$ be an at most countable ordinal, $({\mathcal A}_n)_{n=1}^\infty$ be a sequence of $\sigma$-functionally locally finite covers ${\mathcal A}_n=(A_{i,n}:i\in I_n)$ of $X$ by sets $A_{i,n}\in{\mathcal A}_\alpha(X)$, $((x_{i,n}:i\in I_n))_{n=1}^\infty$ be a sequence of families of points $x_{i,n}\in X$ and $F:X\times Y\to Z$ be a compact-valued /multivalued/ map satisfying the conditions
\begin{enumerate}
  \item[1)] for every $x\in X$, for every $y\in Y$ and for every sequence $(i_n)_{n\in \omega}$ of indexes $i_n\in I_n$ with $x\in A_{i_n,n}$ the sequence $(F(x_{i_n,n},y))_{n\in\omega}$ converges to $F(x,y)$ with respect the Vietoris topology;

  \item[2)] $F^x\in {\rm L}^f_\alpha(Y,Z)$ ($F^x\in {\rm U}^f_\alpha(Y,Z)$) for every $x\in D=\{x_{i,n}:n\in\mathbb N, \,i\in I_n\}$.
\end{enumerate}
Then $F\in {\rm U}^f_{\alpha+1}(X\times Y,Z)$ ($F\in {\rm L}^f_{\alpha+1}(X\times Y,Z)$).
\end{theorem}

\begin{proof} We consider the case of compact-valued map $F$. For every $n\in\omega$ and $i\in I_n$ we put $F_{i,n}=F^{x_{i,n}}$. Let $W\subseteq Z$ be a functionally closed set and $\varphi:Z\to[0,1]$ be a continuous function with $W=\varphi^{-1}(0)$. For every $n\in \omega$ we put $W_n=\varphi^{-1}([0,\tfrac1n])$ and $G_n=\varphi^{-1}([0,\tfrac1n))$. For every $m,n\in \omega$ we put
$$
C_{m,n}=\bigcup_{i\in I_n}\left(A_{i,n}\times F_{i,n}^-(G_m)\right)\,\,\,\mbox{and}\,\, C=\bigcap_{m\in \omega}\bigcup_{n\geq m}C_{n,m}.
$$
Since  $A_{i,n}\in{\mathcal A}_\alpha(X)$ and $F_{i,n}^-(G_m)\in{\mathcal A}_\alpha(Y)$ by condition $2)$, $$A_{i,n}\times F_{i,n}^-(G_m)\in{\mathcal A}_\alpha(X\times Y)$$ for every $m,n\in\omega$ and $i\in I_n$. Now according to Lemma~\ref{lem:1.1}, $C_{m,n}\in {\mathcal A}_\alpha(X\times Y)$. Therefore,
$C\in {\mathcal M}_{\alpha+1}(X\times Y)$.

It remains to show that $C=F^-(W)$. Let $(x_0,y_0)\in F^-(W)$. We fix $m\in \omega$. Notice that $(x_0,y_0)\in F^-(G_m)$. We consider the  neighborhood
$$
O=\{B\subseteq Z:B\cap G_m\ne\emptyset\}
$$
of $F(x_0,y_0)$ with respect the Vietoris topology. According to condition $1)$, there exists $n_0\geq m$ such that for every $n\geq n_0$ if $i\in I_n$ with $x_0\in A_{i,n}$, then $F(x_{i,n},y_0)\in O$, that is, $(x_{i,n},y_0)\in F^-(G_m)$. In particular, for some $i\in I_{n_0}$ we have $x_0\in A_{i,n_0}$ and $y_0\in F_{i,n_0}^-(G_m)$. Therefore, $(x_0,y_0)\in C_{m,n_0}$. Thus, $(x_0,y_0)\in C$.

Now let $(x_0,y_0)\not\in F^-(W)$. Then $$F(x_0,y_0)\subseteq Z\setminus W=\bigcup_{m\in\omega}(Z\setminus W_m).$$ Since $F(x_0,y_0)$ is a compact set, there exists $m_0\in\omega$ such that $F(x_0,y_0)\subseteq Z\setminus W_{m_0}$. We consider the following neighborhood
$$
O_1=\{B\subseteq Z:B\cap W_{m_0}=\emptyset\}
$$
of $F(x_0,y_0)$ with respect the Vietoris topology. According to condition $1)$, there exists $n_0\in\omega$ such that for every $n\geq n_0$
if $x_0\in A_{i,n}$ then $F(x_{i,n},y_0)\in O_1$. Therefore, $F(x_{i,n},y_0)\subseteq Z\setminus W_{m_0}\subseteq Z\setminus W_m$ and $y_0\not\in F_{i.n}^-(G_m)$ for every $m\geq m_0$, $n\geq n_0$ and $i\in I_n$ with $x_0\in A_{i,n}$. This implies that $(x_0,y_0)\not\in C_{n,m}$ for every $n\geq n_0$ and $m\geq m_0$. Thus, $(x_0,y_0)\not\in C$.

Now let $F$ be a multivalued map and $F^x\in {\rm U}^f_\alpha(Y,Z)$ for every $x\in D$. We   argue similarly as in the previous case and   use analogous notations. For every $m,n\in \omega$ we put
$$
C_{m,n}=\bigcup_{i\in I_n}\left(A_{i,n}\times F_{i,n}^+(G_m)\right)
$$
According to Lemma~\ref{lem:1.1}, $C_{m,n}\in {\mathcal A}_\alpha(X\times Y)$ and
$$
C=\bigcap_{m\in \omega}\bigcup_{n\geq m}C_{n,m}\in {\mathcal M}_{\alpha+1}(X\times Y).
$$

Further, we show that $C=F^+(W)$. Let $(x_0,y_0)\in F^+(W)$ and $m\in \omega$. Then $(x_0,y_0)\in F^+(G_m)$. According to condition $1)$, there exist $n\geq m$ and $i\in I_n$ such that $x_0\in A_{i,n}$ and $F(x_{i,n},y_0)\subseteq G_m$. Therefore, $(x_0,y_0)\in C_{m,n}$. Thus, $(x_0,y_0)\in C$.

Now let $(x_0,y_0)\not\in F^+(W)$. Then $F(x_0,y_0)\cap (Z\setminus W)\ne\emptyset$ and there exists $m_0\in\omega$ such that $F(x_0,y_0)\cap (Z\setminus W_{m_0})\ne\emptyset$. According to condition $1)$, there exists $n_0\in\omega$ such that for every $n\geq n_0$ if $x_0\in A_{i,n}$ then $F(x_{i,n},y_0)\cap (Z\setminus W_{m_0})\ne\emptyset$. Therefore, $(x_0,y_0)\not\in C_{n,m}$ for every $n\geq n_0$ and $m\geq m_0$. Thus, $(x_0,y_0)\not\in C$.
\end{proof}

\begin{remark}\label{rem:1} The multivalued map $F:(X,d)\times Y\to Z$ from the Theorem \ref{th:0.1} satisfies conditions $1)-2)$ of Theorem \ref{th:1.2}. For every $u\in D$ and $n\in\omega$ we put
$$
A_{u,n}=A(u)\cap \{v\in X:d(u,v)<\tfrac1n\}\quad{\rm and}\quad x_{u,n}=u.
$$
Then the sequences of families $(A_{u,n}:u\in D)$ and $(x_{u,n}:u\in D)$ provides condition $1)$. Moreover,   condition $2)$ is equivalent to condition $(e)$. Thus, Theorem \ref{th:1.2} generalizes Theorem \ref{th:0.1}.
\end{remark}

\begin{definition}\label{def:4} {\rm
  A topological space $X$ is called {\it (strong) PP-space} if (for any dense set $D\subseteq X$) there exist a sequence $({\mathcal U}_n)_{n=1}^\infty$ of locally finite covers ${\mathcal U}_n=(U_{i,n}:i\in I_n)$ of $X$ and a sequence $((x_{i,n}:i\in I_n))_{n=1}^\infty$ of families of points of $X$ (of $D$) such that for every $x\in X$ and for every neighborhood $U$ of $x$ there exists $n_0\in \omega$ such that for every $n\geq n_0$ and $i\in I_n$ the inclusion $x\in U_{i,n}$ implies the inclusion $x_{i,n}\in U$. }
\end{definition}
    Obviously, every strong PP-space is a PP-space.

For topological spaces $X$, $Y$ and $Z$ and at most countable ordinal $\alpha$ the collection of all multivalued maps $F:X\times Y\to Z$ which are continuous with respect the first variable and are upper (lower) Lebesgue functional class $\alpha$ with respect the second variable we denote by ${\rm CU}^f_\alpha(X,Y,Z)$ (${\rm CL}^f_\alpha(X,Y,Z)$). Analogously, the collection of all multivalued maps $F:X\times Y\to Z$ which are continuous with respect the first variable and such that for some dense in $X$ set $D$ every multivalued map $F^x$ is upper (lower) Lebesgue functional class $\alpha$ we denote by ${\rm \overline{C}U}^f_\alpha(X,Y,Z)$ (${\rm \overline{C}L}^f_\alpha(X,Y,Z)$). The collections ${\rm CU}_\alpha(X,Y,Z)$, ${\rm CL}^f_\alpha(X,Y,Z)$ ${\rm \overline{C}U}^f_\alpha(X,Y,Z)$ and ${\rm \overline{C}L}^f_\alpha(X,Y,Z)$ are introduced analogously.

\begin{corollary}\label{cor:1.3} Let $X$ be a $PP$-space, $Y$ and $Z$ be topological spaces and $\alpha$ be an at most countable ordinal. Then ${\rm CU}^f_\alpha(X,Y,Z)\subseteq {\rm L}^f_{\alpha+1}(X\times Y,Z)$ and ${\rm CL}^f_\alpha(X,Y,Z)\subseteq {\rm U}^f_{\alpha+1}(X\times Y,Z)$.
\end{corollary}

\begin{proof} Let $({\mathcal U}_n)_{n=1}^\infty$ and $((x_{i,n}:i\in I_n))_{n=1}^\infty$ be sequences from Definition \ref{def:4} and $A_{i,n}=U_{i,n}$ for every $n\in\omega$ and $i\in I_n$. It remains to use Theorem \ref{th:1.2}.
\end{proof}

For strongly PP-spaces we can  prove  the following result similarly.

\begin{corollary}\label{cor:2.3} Let $X$ be a strong $PP$-space, $Y$ and $Z$ be topological spaces and $\alpha$ be an at most countable ordinal. Then ${\rm C\overline{U}}^f_\alpha(X,Y,Z)\subseteq {\rm L}^f_{\alpha+1}(X\times Y,Z)$ and ${\rm C\overline{L}}^f_\alpha(X,Y,Z)\subseteq {\rm U}^f_{\alpha+1}(X\times Y,Z)$.
\end{corollary}

\section{Lebesgue classification of multivalued maps of two variables}

We start from the following generalization of Theorem 3.30 from \cite{MMMS}.

\begin{theorem}\label{th:2.1} Let $X$ be a perfect space, $\alpha$ be an at most countable ordinal, $(A_i:i\in I)$ be a locally finite family of sets $A_i$ of additive (multiplicative) class $\alpha$ on $X$. Then the set $A=\bigcup_{i\in I}A_i$ is of additive (multiplicative) class $\alpha$ on $X$.
\end{theorem}

\begin{proof} We argue using the induction on $\alpha$. It is well-known that this statement is true for $\alpha=0$.

Let $(A_i:i\in I)$ be a locally finite family of $F_\sigma$-sets $A_i\subseteq X$ and $((B_{i,n})_{n\in \omega}:i\in I)$ be a sequence of families of closed in $X$ sets $B_{i,n}$ such that $A_i=\bigcup_{n\in\omega}B_{i,n}$ for every $i\in I$. Notice that every family $(B_{i,n}:i\in I)$ is locally finite. Therefore, all sets $B_n=\bigcup_{i\in I}B_{i,n}$ is closed and $A=\bigcup_{n\in\omega}B_{n}$ is an $F_\sigma$-set.

Let $(A_i:i\in I)$ be a locally finite family of $G_\delta$ sets $A_i\subseteq X$ and $((B_{i,n})_{n\in \omega}:i\in I)$ be a sequence of families of open in $X$ sets $B_{i,n}$ such that $A_i=\bigcap_{n\in\omega}B_{i,n}$ for every $i\in I$. For every $i\in I$ we put $F_i=\overline{A_i}$. Clearly, the family $(F_i:i\in I)$ is locally finite and the set $F=\bigcup_{i\in I}F_i$ is closed in $X$. For every $x\in F$ we put $I(x)=\{i\in I:x\in F_i\}$ and $n(x)=|I(x)|$. Moreover, $K_n=\{x\in F_{I(x)}:n(x)>n\}$ for every $n\in\omega$.
Since $(F_i:i\in I)$ is locally finite, every set $K_n$ is closed. We consider the set $C=F\setminus A$ and show that $C$ is an $F_\sigma$-set.

For every $n\in\omega$ we put
$$
C_n=\{x\in C:|n(x)|=n\},
$$
$$
{\mathcal J}_n=\{J\subseteq I:|J|=n\}.
$$
and
$$
C_{J,n}=\{x\in C_n:I(x)=J\}
$$
for every $J\in {\mathcal J}_n$. Show that every family ${\mathcal C}_n=(C_{J,n}:J\in {\mathcal J}_n)$ is locally finite. Fix $x\in X$ and choose a neighborhood $U$ of $x$ in $X$ such that the set $I_1=\{i\in I:U\cap F_i\ne\emptyset\}$ is finite. Then
$$
{\mathcal I}_1=\{J\in {\mathcal J}_n:U\cap C_{J,n}\ne \emptyset\}\subseteq \{J\in {\mathcal J}_n:J\subseteq I_1 \}.
$$
Therefore, ${\mathcal I}_1$ is finite and ${\mathcal C}_n$ is locally finite. Now show that
$$
C_{J,n}=\left(\bigcap_{i\in J}F_i\right)\setminus \left(K_n\cup\bigcup_{i\in J}A_i \right)
$$
for every $n\in \omega$ and $J\in {\mathcal J}_n$. Since $C_{J,n}\subseteq\bigcap_{i\in J}F_i$ and $C_{J,n}\cap(K_n\cup\bigcup_{i\in J}A_i)=\emptyset$, $C_{J,n}\subseteq\left(\bigcap_{i\in J}F_i\right)\setminus \left(K_n\cup\bigcup_{i\in J}A_i \right)$. Conversely, let $x\in \bigcap_{i\in J}F_i$, $x\not\in K_n$ and $x\not\in\bigcup_{i\in J}A_i$. Then $n(x)\geq |J|=n$ and $n(x)\leq n$. Therefore, $n(x)=n$ and $I(x)=J$.
Thus, $x\not\in F_i$ for every $i\in I\setminus J$. Hence,
$$
x\not\in \left(\bigcup_{i\in J}A_i\right)\bigcup\left(\bigcup_{i\in I\setminus J}F_i\right)\supseteq A.
$$
Since $\bigcap_{i\in J}F_i$ and $K_n$ are closed and $\bigcup_{i\in J}A_i$ is a $G_\delta$-set, the $C_{J,n}$ is a $F_\sigma$-set.
Therefore, every $C_n$ is an $F_\sigma$-set as a locally finite union of $F_\sigma$-sets. Thus, $C$ is an $F_\sigma$-set too.

Assume that Lemma is true for all $\alpha<\beta$ where $\beta\geq 1$ is an at most countable ordinal. Let $(A_i:i\in I)$ be a locally finite family of sets of additive class $\beta$ on $X$. Suppose that $\beta=\alpha+1$ for some $\alpha<\omega_1$. Then for every $i\in I$ there exists a sequence $(B_{i,n})_{n\in \omega}$ of sets $B_{i,n}$ of multiplicative class $\alpha$ on $X$ such that $A_i=\bigcup_{n\in\omega}B_{i,n}$ for every $i\in I$. According to the assumption, every set $B_n=\bigcup_{i\in I}B_{i,n}$ belongs to the $\alpha$'th multiplicative class in $X$. Therefore, the set $A=\bigcup_{n\in\omega}B_{n}$ belongs to the $\beta$'th additive class in $X$.

We consider the case of limit ordinal $\beta$. We choose an increasing sequence of ordinals $\alpha_n<\beta$ such that $\sup_{n\in \omega}\alpha_n=\beta$. For every $i\in I$ there exists a sequence $(B_{i,n})_{n\in \omega}$ of sets $B_{i,n}$ of the multiplicative class $\alpha_n$ in $X$ such that $A_i=\bigcup_{n\in\omega}B_{i,n}$ for every $i\in I$. Then every set $B_n=\bigcup_{i\in I}B_{i,n}$ belongs to the  $\alpha$'th  multiplicative class in $X$ and  $A=\bigcup_{n\in\omega}B_{n}$ belongs to the $\beta$'th additive class in $X$.

Now we consider the case of sets of a multiplicative class $\beta\geq 2$. Let $(A_i:i\in I)$ be a locally finite family of sets of the $\beta$'th multiplicative class in $X$, $\beta=\alpha+1$ for some $\alpha<\omega_1$ and $((B_{i,n})_{n\in \omega}:i\in I)$ be a sequence of families of sets $B_{i,n}$ of additive class $\alpha$ in $X$ such that $A_i=\bigcap_{n\in\omega}B_{i,n}$ for every $i\in I$. For every $i\in I$ we put $F_i=\overline{A_i}$. The family $(F_i:i\in I)$
is locally finite. For every $n\in\omega$ and $i\in I$ we put $A_{i,n}=B_{i,n}\cap F_i$. Since $\alpha\geq 1$, every set $A_{i,n}$ belongs to the $\alpha$'th additive class. Then every set $B_n=\bigcup_{i\in I}A_{i,n}$ belongs to the $\alpha$'th  additive class   and $A=\bigcup_{n\in\omega}B_n$.

For a limit ordinal $\beta$ we   argue analogously.
\end{proof}

The following results can be proved similarly as the corresponding results from the previous section.

\begin{theorem}\label{th:2.2}
Let  $X$ be a topological space, $Y$ be a perfect space and $Z$  be a perfectly normal space, $\alpha$ be an at most countable ordinal, $({\mathcal A}_n)_{n=1}^\infty$ be a sequence of $\sigma$-locally finite covers ${\mathcal A}_n=(A_{i,n}:i\in I_n)$ of $X$ by sets $A_{i,n}$ of additive class $\alpha$ on $X$, $((x_{i,n}:i\in I_n))_{n=1}^\infty$ be a sequence of families of points $x_{i,n}\in X$ and $F:X\times Y\to Z$ be a compact-valued (multivalued) map satisfying the conditions
\begin{enumerate}
  \item[1)] for every $x\in X$, for every $y\in Y$ and for every sequence $(i_n)_{n\in \omega}$ of indexes $i_n\in I_n$ with $x\in A_{i_n,n}$ the sequence $(F(x_{i_n,n},y))_{n\in\omega}$ converges to $F(x,y)$ with respect the Vietoris topology;

  \item[2)] $F^x\in {\rm L}_\alpha(Y,Z)$ ($F^x\in {\rm U}_\alpha(Y,Z)$) for every $x\in D=\{x_{i,n}:n\in\mathbb N, \,i\in I_n\}$.
\end{enumerate}
Then $F\in {\rm U}_{\alpha+1}(X\times Y,Z)$ ($F\in {\rm L}_{\alpha+1}(X\times Y,Z)$).
\end{theorem}

\begin{corollary}\label{cor:1.3} Let $X$ be a $PP$-space, $Y$ be a perfect space, $Z$  be a perfectly normal spaces and $\alpha$ be an at most countable ordinal. Then ${\rm CU}_\alpha(X,Y,Z)\subseteq {\rm L}_{\alpha+1}(X\times Y,Z)$ and ${\rm CL}_\alpha(X,Y,Z)\subseteq {\rm U}_{\alpha+1}(X\times Y,Z)$.
\end{corollary}

\begin{corollary}\label{cor:2.3} Let $X$ be a strong $PP$-space, $Y$ and $Z$ be topological spaces and $\alpha$ be an at most countable ordinal. Then ${\rm C\overline{U}}_\alpha(X,Y,Z)\subseteq {\rm L}_{\alpha+1}(X\times Y,Z)$ and ${\rm C\overline{L}}_\alpha(X,Y,Z)\subseteq {\rm U}_{\alpha+1}(X\times Y,Z)$.
\end{corollary}

\begin{remark} Since every metrizable space is a strong $PP$-space, Theorem \ref{th:2.2} generalizes Theorem \ref{th:0.2}. Moreover, Theorem \ref{th:2.2} generalizes Theorem \ref{th:0.1} (it is enough to argue analogously as in Remark \ref{rem:1}).
\end{remark}

The following example shows that the condition of map to be compact-valued is essential in Theorems \ref{th:1.2} and \ref{th:2.2}.

\begin{proposition} \label{pr:2}
There exists a separately continuous lower semi-continuous map $F:[0,1]^2\to[0,1]$ which is upper (functionally) nonmeasurable.
\end{proposition}

\begin{proof}
Let $A\subseteq[0,1]$ be a non-measurable set. We consider a continuous function $g:[0,1]^2\to[0,1]$, $g(x,y)=\tfrac{2(x+1)(y+1)}{(x+1)^2+(y+1)^2}$, and
a multivalued maps $F:[0,1]^2\to[0,1]$,
 \begin{equation*}
 F(x,y)=\left\{\begin{array}{ll}
                        [0,g(x,y)], & (x,y)\in [0,1]^2\setminus \{(z,z):z\in A\}\\
                        {[}0,1), & (x,y)\in \{(z,z):z\in A\}.
                       \end{array}
 \right.
 \end{equation*}
We put $\Delta= \{(x,x):x\in [0,1]\}$. Since the function $g$ is continuous, the map $F$ is separately continuous at every point of the set $[0,1]^2\setminus \{(z,z):z\in A\}$ and jointly continuous at every point of the set $[0,1]^2\setminus\Delta$. Since the map $H(x,y)=[0,g(x,y)]$ is continuous and $H^-(G)=F^-(G)$ for every open set $G\subseteq[0,1]$, the map $F$ is lower semi-continuous. Moreover, $F(x,y)\subseteq [0,1)\subseteq F(z,z)$ for every $(x,y)\in [0,1]^2\setminus\Delta$ and $z\in[0,1]$. Therefore, $F$ is separately upper semi-continuous at every point of the set $\Delta$. Thus, $F$ is a separately continuous lower semi-continuous map. According to Example \ref{ex:1}, the restriction $F|_\Delta$ is upper nonmeasurable. Hence, $F$ is upper nonmeasurable too.
\end{proof}

\bibliographystyle{amsplain}

\end{document}